\documentclass[12pt]{amsart}
\usepackage{verbatim}
\usepackage{amssymb, amsmath}
\usepackage{graphicx}
\usepackage{caption}
\usepackage{subcaption}
\usepackage{enumerate}

\usepackage{color}

\begin{document}
\newtheorem{thm}{Theorem}[section]
\newtheorem*{thm*}{Theorem}
\newtheorem{lem}[thm]{Lemma}
\newtheorem{prop}[thm]{Proposition}
\newtheorem{cor}[thm]{Corollary}
\newtheorem*{conj}{Conjecture}
\newtheorem{proj}[thm]{Project}
\newtheorem{question}[thm]{Question}
\newtheorem{rem}{Remark}[section]
\newtheorem{notation}[thm]{Notation}
\newtheorem{proposition}[thm]{Proposition}

\theoremstyle{definition}
\newtheorem*{defn}{Definition}
\newtheorem*{remark}{Remark}
\newtheorem{exercise}{Exercise}
\newtheorem*{exercise*}{Exercise}

\numberwithin{equation}{section}

\newcommand{\rad}{\operatorname{rad}}

\newcommand{\Z}{{\mathbb Z}} 
\newcommand{\Q}{{\mathbb Q}}
\newcommand{\R}{{\mathbb R}}
\newcommand{\C}{{\mathbb C}}
\newcommand{\N}{{\mathbb N}}
\newcommand{\FF}{{\mathbb F}}
\newcommand{\fq}{\mathbb{F}_q}
\newcommand{\rmk}[1]{\footnote{{\bf Comment:} #1}}

\renewcommand{\mod}{\;\operatorname{mod}}
\newcommand{\ord}{\operatorname{ord}}
\newcommand{\TT}{\mathbb{T}}
\renewcommand{\i}{{\mathrm{i}}}
\renewcommand{\d}{{\mathrm{d}}}
\renewcommand{\^}{\widehat}
\newcommand{\HH}{\mathbb H}
\newcommand{\Vol}{\operatorname{vol}}
\newcommand{\area}{\operatorname{area}}
\newcommand{\tr}{\operatorname{tr}}
\newcommand{\norm}{\mathcal N} 
\newcommand{\intinf}{\int_{-\infty}^\infty}
\newcommand{\ave}[1]{\left\langle#1\right\rangle} 
\newcommand{\Var}{\operatorname{Var}}
\newcommand{\Prob}{\operatorname{Prob}}
\newcommand{\sym}{\operatorname{Sym}}
\newcommand{\disc}{\operatorname{disc}}
\newcommand{\CA}{{\mathcal C}_A}
\newcommand{\cond}{\operatorname{cond}} 
\newcommand{\lcm}{\operatorname{lcm}}
\newcommand{\Kl}{\operatorname{Kl}} 
\newcommand{\leg}[2]{\left( \frac{#1}{#2} \right)}  
\newcommand{\Li}{\operatorname{Li}}

\newcommand{\sumstar}{\sideset \and^{*} \to \sum}

\newcommand{\LL}{\mathcal L} 
\newcommand{\sumf}{\sum^\flat}
\newcommand{\Hgev}{\mathcal H_{2g+2,q}}
\newcommand{\USp}{\operatorname{USp}}
\newcommand{\conv}{*}
\newcommand{\dist} {\operatorname{dist}}
\newcommand{\CF}{c_0} 
\newcommand{\kerp}{\mathcal K}

\newcommand{\Cov}{\operatorname{cov}}
\newcommand{\Sym}{\operatorname{Sym}}

\newcommand{\Ht}{\operatorname{Ht}}

\newcommand{\E}{\operatorname{\mathbb E}} 
\newcommand{\sign}{\operatorname{sign}} 
\newcommand{\meas}{\operatorname{meas}} 
\newcommand{\length}{\operatorname{length}} 

\newcommand{\divid}{d} 

\newcommand{\GL}{\operatorname{GL}}
\newcommand{\SL}{\operatorname{SL}}
\newcommand{\re}{\operatorname{Re}}
\newcommand{\im}{\operatorname{Im}}
\newcommand{\res}{\operatorname{Res}}
 \newcommand{\eigen}{\Lambda} 
\newcommand{\tens}{\mathbf t} 
\newcommand{\diam}{\operatorname{diam}}
\newcommand{\fixme}[1]{\footnote{Fixme: #1}}
 \newcommand{\EWp}{\mathbb E^{\rm WP}} 
\newcommand{\orb}{\operatorname{Orb}}
\newcommand{\supp}{\operatorname{Supp}}
\newcommand{\mmfactor }{\textcolor{red}{c_{\rm Mir}}}
\newcommand{\Mg}{\mathcal M_g} 
\newcommand{\MCG}{\operatorname{Mod}} 
\newcommand{\Diff}{\operatorname{Diff}} 
\newcommand{\If}{I_f(L,\tau)}
\newcommand{\SigGOE}{\Sigma^2_{\rm GOE}}

\newcommand{\Nc}{\mathcal{N}}  
\newcommand{\Rpos}{\R_{>0}}
\newcommand{\Rnneg}{\R_{\geq 0}}
\newcommand{\vlim}{ \overset{v}{\rightarrow}}
\newcommand{\dlim}{ \overset{d}{\rightarrow}}
\newcommand{\Pois}{\operatorname{Pois}}

\newcommand {\Vc} {\mathcal{V}}
\newcommand{\new}[1]{{\color{blue} #1}}
\newcommand{\EMP}{\E_{\Pois}}
\newcommand{\diag}{\operatorname{Diag}}
\newcommand{\off}{\operatorname{Off}}
\newcommand {\Lc} {\mathcal{L}}
\newcommand{\Mp}{\pazocal{M}}

\title[Almost sure GOE fluctuations]{Almost sure GOE fluctuations of energy levels for hyperbolic surfaces of high genus }

\author{Ze\'ev Rudnick  and Igor Wigman}
\address{School of Mathematical Sciences, Tel Aviv University, Tel Aviv 69978, Israel}
\email{rudnick@tauex.tau.ac.il}
\address{Department of Mathematics, King's College London, UK}
\email{igor.wigman@kcl.ac.uk}

\thanks{ This research was supported by the European Research Council (ERC) under the European Union's Horizon 2020 research and innovation programme (grant agreement No. 786758) and by the Israel Science Foundation (grant No. 1881/20).  }

\begin{abstract}
We study the variance of a linear statistic of the Laplace eigenvalues on a hyperbolic surface, when the surface varies over the moduli space of all surfaces of fixed genus, sampled at random according to the Weil-Petersson measure. The ensemble variance of the linear statistic   was recently shown to coincide with that of the corresponding statistic in the Gaussian Orthogonal Ensemble (GOE) of random matrix theory, in the double limit of first taking large genus and then shrinking size of the energy window. In this note we show that in this same limit,
the energy variance for a typical surface is close to the GOE result, a feature called ``ergodicity'' in the random matrix theory literature.
\end{abstract}

\date{\today}
\maketitle

 \section{Introduction  }

 \subsection{Statement of results}
 Let $X$ be a compact hyperbolic surface of genus $g\ge 2$, $\lambda_{j}=1/4+r_{j}^{2}$ the Laplace eigenvalues on $X$. Let $f$ be a test function so that its Fourier transform $\widehat{f}$ is even, smooth and compactly supported, $L\geq 1$, $\tau>0$, and define the smooth linear statistic
\begin{equation*}
N_{L,\tau}(X)  := \sum\limits_{j\ge 0}f(L(r_{j}-\tau))+f(L(r_{j}+\tau))  ,
\end{equation*}
effectively counting eigenvalues in a window of size $\approx \tau/L$ around $\tau^2$, as in \cite{RudnickCLT, RGOE, Naud, RWCLT}.
We write $N_{f,L,\tau}(X)$ as a sum of a smooth and fluctuating terms
\begin{equation*}
N_{L,\tau}(X) = \overline{N} + \widetilde{N}_{L,\tau}(X)  ,
\end{equation*}
where
\[
\overline{N} := 2(g-1)\int\limits_{-\infty}^{\infty}f(L(r-\tau))r\tanh(\pi r)dr\sim \frac{2\tau (g-1)}{L}\int\limits_{-\infty}^{\infty}f(x)dx.
\]

In \cite{RGOE} it was shown that when averaged over the moduli space $\Mg$ of surfaces of genus $g$ with respect
to the Weil-Petersson measure, the variance of $\widetilde{N}_{L,\tau}(X) $ matches that of the corresponding statistic in the Gaussian Orthogonal Ensemble (GOE) of random matrix theory,  in the double limit, taking   $g\to \infty$,
and then   $L\to \infty$: 
\begin{equation}\label{eq:RGOE}
\lim_{L\to \infty}  \lim_{g\to \infty}\EWp_g\left(  \left| \widetilde{N}_{L,\tau} -\EWp_g( \widetilde{N}_{L,\tau}) \right|^2 \right)
= \SigGOE(f)
\end{equation}
where $\SigGOE(f) = 2\intinf |x|\^f(x)^2dx$.

We now consider the energy average.
Let $\omega $ be a non-negative,  even weight function, normalized by $\intinf \omega(x)dx=1$, with Fourier transform $\widehat{\omega}$  smooth and supported in $[-1,1]$. For $T>0$ define an averaging operator
\begin{equation}
\label{eq:averageT op def}
\E_T[F]:=\frac 1T\intinf F(\tau)\omega\left(\frac \tau T \right) dt
\end{equation}
and a corresponding variance
\begin{equation}
\label{eq:varianceT op def}
\Var_T(F)= \E_T\left[\left| F-\E_T(F) \right|^2 \right].
\end{equation}

Denote by $\Vc_{T,L}(X) $  the energy variance of $\widetilde{N}_{L,\tau}(X) $, thought of as a random variable on $\Mg$:
\begin{equation*}
 \Vc_{T,L}(X) :=\Var_{T}(\widetilde{N}_{L,\tau}(X)) .
\end{equation*}
Following Berry \cite{Berry1986}, it is believed that for {\em generic}
surfaces\footnote{Arithmetic surfaces are exceptional, see \cite{BGGS, LS}.}
 $X\in \Mg$,  for fixed genus\footnote{Genus $g=2$ requires a modification due to the presence of the hyperelliptic involution.} $g>2$, the energy variance $ \Vc_{T,L}(X)$  will converge to the GOE variance $\SigGOE(f) $ when $T\to \infty$, and $L \to \infty$ but $L=o(T)$.
  Our  principal result supports  this   if we first take the large genus limit $g\to \infty$, and then the high energy limit $T\to\infty$, while restricting to energy windows $L\to \infty$ with $L=o(\log T)$, more precisely that in this limit the random variable $\Vc_{T,L}(X) $ converges in distribution to the constant $\SigGOE(f)$:
\begin{thm}
\label{thm:almost sure GOE}
For every $\epsilon>0$,
\begin{equation}
\label{eq:dev>eps -> 0 dbl lim}
\lim\limits_{\substack{L,T\rightarrow\infty\\ L=o(\log{T})}}\limsup\limits_{g\rightarrow\infty} \Prob^{WP}_{g}\left(\left|\Vc_{L,T} - \Sigma_{GOE}^{2}(f) \right| > \epsilon  \right) = 0 .
\end{equation}
\end{thm}

This may be viewed as an analogue of a feature called ``ergodicity'' in random matrix theory: In the limit of large matrix size, the energy variance equals the ensemble variance for almost all matrices, see \cite{Pandey}.

\subsection{Method of proof}

Using the Selberg trace formula, we  have \cite{RGOE}
\[
N_{L,\tau}(X)=   \overline{N} +\widetilde{N}_{L,\tau}(X)
\]
with
\[
 \overline{N} =2(g-1)\int\limits_{-\infty}^{\infty}f(L(r-\tau))r\tanh(\pi r)dr\sim \frac{2\tau (g-1)}{L}\int\limits_{-\infty}^{\infty}f(x)dx,
\]
and
\begin{equation}
\label{eq:N tilde osc ident}
 \widetilde{N}_{L,\tau}(X) =\frac{2}{L}\sum\limits_{\ell_\gamma \in \mathcal L_g(X) }\sum\limits_{k\ge 1}\frac{\ell_{\gamma}}{\sinh(k\ell_{\gamma}/2)}\widehat{f}\left(\frac{k\ell_{\gamma}}{L}\right)\cos(\tau k\ell_{\gamma})
\end{equation}
 where the summation is over the primitive length spectrum
 $$\mathcal L_g(X)\subset(0,\infty),$$
 the set of lengths of primitive non-oriented closed geodesics of $X$.

We view the primitive length spectrum as a random point set $\Lc_{g}$, parameterized by the random variable $X\in \Mg$, that is a random point process, and the variance $\Vc_{T,L}(X)$ is an application $\varphi(\Lc_{g})$ of a functional $\varphi$ on $\Lc_{g}$, which is continuous in a suitable topology.
A fundamental result of Mirzakhani and Petri \cite{MP} shows that the point processes $\mathcal L_g$ converge, in distribution as $g\to \infty$, to a Poisson point process, denoted $\Lc_{\infty}$ on the positive reals, with intensity measure $$d\nu_{MP}(x) = \frac{2\sinh^2 x}{x} dx.$$

From the theory of point processes, it follows that $\Vc_{T,L}(X)=\varphi(\Lc_{g})$ converges, in distribution as $g\rightarrow\infty$ to the random variable $\Vc_{T,L}^\infty= \varphi(\Lc_{\infty})$, see Proposition \ref{prop:VTL->Vinf}. To show
Theorem~\ref{thm:almost sure GOE}, we argue that it suffices to prove that $\Vc_{T,L}^\infty$, converges in distribution to the constant $\SigGOE(f)$ as $L,T\to \infty$ with $L=o(\log T)$ (see Theorem \ref{thm:VTL conv constant}).
The proof of the latter result takes up the bulk of this paper, and is done in
\S~\ref{sec:proof of conv to constant}.

\section{Outline of the proof of the main result}

We rewrite \eqref{eq:N tilde osc ident} as
\begin{equation}
\label{eq:N tilde osc expr H}
 \widetilde{N}_{L,\tau}(X)
 =\sum_{\ell\in \Lc_{g}}  H_{L,\tau}(\ell),
\end{equation}
where
\begin{equation}
\label{eq:HLt def}
H_{L,\tau}(x)= 2\frac xL \sum_{k\geq 1} \frac{\^f(\frac{kx}{L})  \cos(kx \tau)}{\sinh(kx/2)}
\end{equation}
that we think of as a functional of the length spectrum, that is considered as a random point process, as in the following notation.

\begin{notation}
\label{not:length spec}

\begin{enumerate}[i.)]

\item For $g\ge 1$ let $\Lc_{g}=\Lc_{g}(X)=\{\ell_{j}\}_{j\ge 1}$ be the (random) primitive unoriented length spectrum of $X\in \Mg$,
considered as a random point process on $\R_{>0}$.

\item Let $\Lc_{\infty} = \{\ell_{j}\}_{j\ge 1}$ be the Poisson point process of intensity
\begin{equation}
\label{eq:nu dens def}
\nu_{MP}(dt):=\frac{2\sinh(t/2)^{2}}{t}dt.
\end{equation}

\end{enumerate}
\end{notation}

It was proved by Mirzakhani-Petri ~\cite{MP}, that $\Lc_{g}$ converges, as a sequence of point processes, to $\Lc_{\infty}$, see \S~\ref{sec:background PP} below. It is then natural to consider the analogue of $\widetilde{N}_{L,\tau}(X)$, with $\Lc_{g}$ replaced by $\Lc_{\infty}$. That is, define the random variable
\begin{equation}
\label{eq:SLtau def}
S_{L,\tau} = \sum_{\ell\in \Lc_{\infty}}  H_{L,\tau}(\ell),
\end{equation}
where $H_{L,\tau}(x)$ is given by \eqref{eq:HLt def}. We recall the averaging operator $\E_{T}[\cdot]$, as in \eqref{eq:averageT op def},
and apply the corresponding variance operator $\Var_T(\cdot)$ as in \eqref{eq:varianceT op def}, on $S_{L,\tau}$:
$$\Vc_{T,L}^{\infty}:=\Var_{T}(S_{L,\tau})=\E_T\left[ \left|N_{L,\tau}(X)-\E_T\left(N_{L,\tau}\left(X\right)\right| \right)^2\right].$$
We aim to express $\Vc_{T,L}^{\infty}$ directly in terms of $\Lc_{\infty}$.

To this end, squaring \eqref{eq:SLtau def}, after some simple manipulations with the resulting expression, gives
\begin{equation}\label{semiclassical expression for V}
 \Vc_{T,L}^{\infty} = \left( \frac 2L \right)^2\sum_{\ell_1,\ell_2\in \Lc_{\infty}}
 \sum_{k_1,k_2\geq 1} \frac { H_L(k_1\ell_1)  H_L(k_2\ell_2) }{k_1  k_2 }
 U_T(k_1\ell_1,k_2\ell_2)
\end{equation}
with
\begin{equation}\label{eq:HL def}
H_L(x) =  \frac{x\^f\left( \frac xL \right)  }{\sinh(x/2)}
\end{equation}
and
\begin{equation}
\label{def of UT}
\begin{split}
U_T(x,y) &=\E_T \left[\cos\left(\tau x\right) \cos\left(\tau y\right) \right] -\E_T \left[\cos\left(\tau x \right) \right] \cdot
\E_T \left[\cos \left(\tau y\right) \right]
\\
&=  \frac 12   \^w\left( T\left( x-y\right)\right)+ \frac 12 \^w\left(T\left(x+y\right)\right)-\^w\left(Tx\right) \cdot \^w\left(Ty\right) .
\end{split}
\end{equation}
Note that $U_T(x,y)$ is real valued since $w(x)$ is real valued and even. Starting from \eqref{eq:N tilde osc expr H} in place of
\eqref{eq:SLtau def}, and likewise yields
\begin{equation}
\label{semiclassical expression for Vg}
\Vc_{T,L}(X) = \left( \frac 2L \right)^2\sum_{\ell_1,\ell_2\in \Lc_{g}} \sum_{k_1,k_2\geq 1} \frac { H_L(k_1\ell_1)  H_L(k_2\ell_2) }{k_1  k_2 }
 U_T(k_1\ell_1,k_2\ell_2).
\end{equation}

Below we will be able to deduce Theorem \ref{thm:almost sure GOE} from the following two results: Proposition \ref{prop:VTL->Vinf}, proved at the end of \S~\ref{sec:background PP}, asserting that the random variables $\Vc_{T,L}(X)$ converge in distribution to $\Vc_{T,L}^{\infty}$, and
Theorem \ref{thm:VTL conv constant}, proved along \S~\ref{sec:proof of conv to constant}, asserting that, in the regime of Theorem \ref{thm:almost sure GOE}, $V_{T,L}^{\infty}$ converge, in distribution, to the constant $\SigGOE(f)$.

\begin{proposition}
\label{prop:VTL->Vinf}
For every $T,L>0$, the random variables $\Vc_{T,L}=\Vc_{T,L}(X)$, with $X\in\Mg$ random uniform w.r.t. the WP measure,
converge, in distribution as $g\rightarrow\infty$, to the random variable $\Vc_{T,L}^{\infty}$.
\end{proposition}

\begin{thm}
\label{thm:VTL conv constant}
One has
\[
\lim_{\substack{L,T\to \infty\\L=o(\log T)}} \EMP \left[ \left| \Vc_{T,L}^\infty-\SigGOE(f) \right| \right] = 0 .
\]
\end{thm}

Theorem \ref{thm:VTL conv constant} with Markov's inequality implies that the probability that
$\Var_T(S_{L,\tau} ^2)$ is far from $\SigGOE(f)$ vanishes: for every $\epsilon>0$,
\begin{equation}
\label{eq:VTL Pois close GOE}
\lim_{\substack{L,T\to \infty\\L=o(\log T)}}\Prob_{\operatorname{Pois}} \left( \left|  \Vc_{T,L}^\infty   -\SigGOE(f) \right|>\epsilon \right) =0.
\end{equation}

\section{Background on point processes}
\label{sec:background PP}

Let $\Nc$ be the space of measures $\mu$ on $\R_{\ge 0}$ s.t. $\mu(0)=0$, for every $t\ge 0$, $\mu(t)\in \Z_{\ge 0}$, where $\mu(t):=\mu((0,t])$, which is assumed to be right-continuous at $0$.
A measure $\mu\in\Nc$ has a representation
\begin{equation}
\label{eq:mu=sum delta}
\mu=\sum\limits_{j=1}^{k}\delta_{\xi_{j}},
\end{equation}
where $0\le k\le +\infty$, and $0<\xi_{1}\le \xi_{2}\le\ldots$, and sometimes it is convenient to regard $\mu$ as a discrete multi-set
\begin{equation}
\label{eq:mu Nc discrete set}
\mu=\{\xi_{j}\}_{j\ge 1}\subseteq \R
\end{equation}
Given a sequence $\{\mu_{n}\}\subseteq \Nc$ and $\mu\in\Nc$, we say that $\mu_{n}$ {\em vaguely converges} to $\mu$, denoted $\mu_{n}\rightarrow\mu$, if for all $t\in\R_{\ge 0}$ so that $\mu(\cdot)$ is continuous at $t$,
\begin{equation}
\label{eq:mun->mu vague}
\mu^{n}(t)\rightarrow\mu(t).
\end{equation}
Equivalently, for every non-negative continuous test function of compact support $f:\R_{\ge 0}\rightarrow\R$, one has
$$ \int\limits_{0}^{\infty}f(x)\mu^{n}(dx) \rightarrow   \int\limits_{0}^{\infty}f(x)\mu(dx) .$$

A point process (in a somewhat restrictive sense, sufficient for our purposes)
is a {\em random} element $\mu\in\Nc$ (w.r.t. the vague topology on $\Nc$, generating the Borel $\sigma$-algebra on $\Nc$).
(Equivalently, it is a probability measure on $\Nc$.)
We will regard a point process as a random element of $\Nc$.


We say that a point process $\mathbf{N}$ is Poisson with intensity measure $\nu$ on $\R_{>0}$, if for every Borel set $B\subseteq \R$, the distribution of $\mathbf{N}(B)$ is Poisson with parameter $\nu(B)$, and the random variables $\mathbf{N}(B_{1}),\ldots,\mathbf{N}(B_{k})$ are independent for every $k\ge 2$, and any choice of {\em disjoint} Borel sets $B_{1},\ldots, B_{k}$.
A sequence of point processes $\mathbf N_n$ on $\R_{>0}$ is said to converge (in distribution) to a point process $\mathbf N$
(written $\mathbf N_n \dlim \mathbf N$) if the sequence of random vectors $(\mathbf N_n(B_1),\dots, \mathbf N_n(B_k))$ converges in distribution to the random vector $(\mathbf N(B_1),\dots, \mathbf N(B_k))$ for all $k\geq 1$ and Borel sets $B_i$ with boundaries satisfying $\mathbf N(\partial B_i)=0$ almost surely for all $i$. The  convergence of $\mathbf N_n$ to $\mathbf N$ implies
~\cite[Theorem 23 on p. 521]{Grandell} that the random element $\mathbf N_n$ converges, in distribution, to $\mathbf N$.


For a measure $\mu\in \Nc$ as in \eqref{eq:mu=sum delta} with $k\ge 1$ finite, and $1\le m\le k$,
one defines ~\cite[p. 28]{LaPe}
the $m$'th {\em factorial measure} $\mu^{(m)}$ of $\mu$ on $\R_{\ge 0}^{(m)}$ as
\begin{equation*}
\mu^{(m)} = \sum\limits_{i_{1},\ldots i_{m} \text{ distinct}}\delta_{(x_{i_{1}},\ldots x_{i_{m}})}.
\end{equation*}
Next, for a point process $\eta$ we define
~\cite[Definition 4.9]{LaPe} the $m$'th factorial {\em moment} measure as $$\alpha_{m}(C)=\alpha_{\eta;m}(C):= \E\left[\eta^{(m)}(C)\right]$$ for
$C\subseteq \R_{\ge 0}^{m}$ measurable (Borel).
Then the $m$'th factorial measure of a Poisson point process $\eta$ with intensity $\lambda$ is ~\cite[Corollary 4.10]{LaPe} a Poisson point process with the intensity $\lambda^{m}$, the $m$'th product measure of $\lambda$ on $\R_{\ge 0}^{m}$, that is also the $m$'th factorial moment measure of $\eta$.

In what follows use the identification \eqref{eq:mu Nc discrete set} of a measure $\mu\in\Nc$. Thus the factorial moment measure of a {\em proper} point process $\eta$ (that includes any Poisson point process) computes the correlations of $\{\xi_{j}\}$ in the sense that
for a function $f:\R_{\ge 0}^{m}\rightarrow\R$,
\begin{equation}
\label{eq:exp Camp gen}
\E \Big[ \sum_{\substack{(\xi_{1},\dots,\,\xi_{m})\in \eta^{m}\\ \text{ distinct}}} f(\xi_{1},\dots,\xi_{m}) \Big]
= \int\limits_{\R_{\ge 0}^{m}}f(\xi_{1},\ldots,\xi_{m})d\alpha_{m}(\xi_{1},\dots,\xi_{m}) ,
\end{equation}
see ~\cite[Definition 4.9 and immediately after]{LaPe}. In particular, for $m=1$, this is Campbell's formula ~\cite[Proposition 2.7]{LaPe}.


We conclude this section with the following observation. By comparing \eqref{semiclassical expression for V} to
\eqref{semiclassical expression for Vg}, we may express both as
\begin{equation*}
\Vc_{T,L}(X) = \varphi_{T,L}(\Lc_{g});\;\; \Vc_{T,L}^{\infty} = \varphi_{T,L}(\Lc_{\infty}),
\end{equation*}
with the functional $\varphi_{T,L}:\Nc\rightarrow\R$, defined as
\begin{equation}
\label{eq:varphi func def}
\varphi_{T,L}(\mu) :=  \left( \frac 2L \right)^2\sum_{\ell_1,\ell_2\in \mu} \sum_{k_1,k_2\geq 1} \frac { H_L(k_1\ell_1)  H_L(k_2\ell_2) }{k_1  k_2 } U_T(k_1\ell_1,k_2\ell_2).
\end{equation}
This is easily seen to be continuous, see ~\cite[Lemma 2.1]{RWCLT}.

\begin{proof}[Proof of Proposition \ref{prop:VTL->Vinf}]
We interpret the result of Mirzakhani-Petri ~\cite[Theorem 4.1]{MP} as the convergence, as $g\rightarrow\infty$, of the point processes $\Lc_{g}$ to the Poisson point process $\Lc_{\infty}$ (Notation \ref{not:length spec}), and recall that it implies that the random element $\Lc_{g}\in \Nc$ converges, in distribution, to $\Lc_{\infty}\in\Nc$. Hence, an application of the Continuous Mapping theorem, with the continuous functional $\varphi:\Nc\rightarrow\R$ as in \eqref{eq:varphi func def}, yields that, as $g\rightarrow\infty$, the random variables
$$\Vc_{T,L}(X) =\varphi(\Lc_{g})$$ converge, in distribution, to
$$ \Vc_{T,L}^{\infty} = \varphi(\Lc_{\infty}),$$ that is, the statement of Proposition \ref{prop:VTL->Vinf}.

\end{proof}

\section{GOE fluctuations for the Poisson model: Proof of Theorem \ref{thm:VTL conv constant}}
\label{sec:proof of conv to constant}

We need to compute the expected value of $\left|\Vc_{T,L}^\infty-\SigGOE(f) \right|$ with $\Vc_{T,L}^\infty$ given by
\eqref{semiclassical expression for V} with $\ell_1,\ell_2$ selected according to the Poisson point process. We write \eqref{semiclassical expression for V} according to whether $\ell_1=\ell_2$ (the diagonal) and $\ell_1\neq \ell_2$:
\[
 \Vc_{T,L}^\infty= \sum_{\ell_1=\ell_2} + \sum_{\ell_1\neq \ell_2} = \diag+ \off
\]
and then we have
\[
 \EMP \left[\left|\Vc_{T,L}^\infty-\SigGOE(f) \right|  \right]  \leq  \EMP \left[\left|\diag-\SigGOE(f) \right|  \right]  + \EMP\left[ \left| \off \right| \right] .
\]
We will show that both terms tend to zero as $L,T\to \infty$, $L=o(\log T)$.


\subsection{Bounding $\EMP\left[ \left| \off \right| \right]$}

 \begin{lem}
\label{lem:Ioffdiag vanish exp}
There is some $C_{0}=C_0(f)>0$ so that
\begin{equation}
\label{eq:Ioffdiag vanish exp}
\EMP\left[ \left| \off \right| \right]\ll \frac{e^{2C_{0}L}}{TL},
\end{equation}
the constant involved in the `$\ll$'-notation absolute.
\end{lem}
\begin{proof}
We apply the correlation formula \eqref{eq:exp Camp gen} with $m=2$ to obtain
\[
f(\xi_{1},\xi_{2}) =
\frac{4}{L^{2}}\sum\limits_{k_{1},k_{2}\ge 1} \frac{|H_{L}(k_{1}\xi_{1})||H_{L}(k_1\xi_1))|}{k_1   k_2}
  | U_T(k_{1}\xi_{1},k_{1}\xi_{1}) |
\]
and intensity $d\nu_{MP}(x) = 2\sinh^2(x/2)/x \;dx$, to yield that
\begin{multline}\label{eq:IOFF- def}
\EMP\left[ \left| \off \right| \right]   = \int\limits_{0}^{\infty}\int\limits_{0}^{\infty}f(x,y)d\nu_{MP}(x)d\nu_{MP}(y)
\\ =\frac{4}{L^{2}}\sum\limits_{k_{1},k_{2}\ge 1}\int\limits_{0}^{\infty}\int\limits_{0}^{\infty}
\frac{| H_{L}(k_{1}x)H_{L}(k_{2}y)|}{k_1k_2}
 |U_T(k_1x,k_2y)| d\nu_{MP}(x)d\nu_{MP}(y)
\\ =
\frac{16}{L^{2}}\sum\limits_{k_{1},k_{2}\ge 1}\int\limits_{0}^{\infty}\int\limits_{0}^{\infty}
\widehat{f}\left(\frac{k_{1}x}{L} \right)  \widehat{f}\left(\frac{k_{2}y}{L} \right)  \cdot
\frac{\sinh(x/2)^{2}\cdot \sinh(y/2)^{2}}{\sinh\left( \frac{k_{1}x}{2} \right) \cdot \sinh\left( \frac{k_{2}y}{2} \right)}
\times\\
\times |U_T(k_1x,k_2y)|
dxdy
\end{multline}
by \eqref{eq:HL def}.
In what follows we show that when $L=o(\log{T})$, this  expression vanishes.

We note that   $\widehat{f}$ being compactly supported forces that in the range of the integral
one has $0< x\ll \frac{L}{k_{1}}$, so that $\sinh(x/2) \le e^{C_{0}\cdot L}$ for some absolute $C_{0}>0$, and likewise $\sinh(y/2) \le e^{C_{0}\cdot L}$. Since for $k\ge 1$ and $t\in \R$
\begin{equation*}
\frac{\sinh(t)}{\sinh(kt)} \le \frac{1}{k},
\end{equation*}
we have
$$\frac{\sinh(x/2)^{2}\cdot\sinh(y/2)^{2}}{\sinh\left( \frac{k_{1}x}{2} \right) \cdot \sinh\left( \frac{k_{2}y}{2} \right)} \le e^{2C_{0}\cdot L}\cdot \frac{1}{k_{1}k_{2}}.$$ Given $x$ in the domain of the integration (as above, $x\ll \frac{L}{k_{1}}$), the compact support of $\widehat{\omega}$  forces that $y$ is contained in an interval of length $\frac{1}{k_{2}T}$
(from \eqref{def of UT}). Therefore, taking into account the boundedness of both $\widehat{f}$ and $\widehat{w}$ (hence of $U_T$), we have
\begin{equation}
\label{eq:Ioff diag vanish single term}
\begin{split}
\frac{1}{L^{2}}\int\limits_{0}^{\infty}\int\limits_{0}^{\infty}
&\widehat{f}\left(\frac{k_{1}x}{L} \right)  \widehat{f}\left(\frac{k_{2}y}{L} \right)
\frac{\sinh(x/2)^{2} \sinh(y/2)^{2}}{\sinh\left( \frac{k_{1}x}{2} \right)  \sinh\left( \frac{k_{2}y}{2} \right)}
|U_T(k_1x,k_2y)|
dxdy
\\&\ll \frac{1}{L^{2}}\cdot \frac{e^{2C_{0}L}}{k_{1}k_{2}}\cdot \frac{1}{Tk_{2}}\cdot \frac{L}{k_{1}} = \frac{e^{2C_{0}L}}{TL}\cdot \frac{1}{k_{1}^{2}k_{2}^{2}}.
\end{split}
\end{equation}
The bound \eqref{eq:Ioffdiag vanish exp} of
Lemma \ref{lem:Ioffdiag vanish exp} finally follows upon summing up \eqref{eq:Ioff diag vanish single term}
for $k_{2}, k_{1}\ge 1$ 
and substituting it into \eqref{eq:IOFF- def}.
\end{proof}

\subsection{Bounding the diagonal term}
We now want to bound the term $ \EMP \left[\left|\diag-\SigGOE(f) \right|  \right] $. Recall
\[
\diag=   \frac 4{L^2}\sum_x \sum_{k_1,k_2\geq 1} \frac { H_L(k_1x)   H_L(k_2x) }{k_1  k_2 } U_T(k_1x,k_2x)
\]

We separate out the contribution of the pair $(k_1,k_2)=(1,1)$ and the rest and use
\begin{multline}\label{separate k1+k2 geq 3}
\EMP \left[\left|\diag-\SigGOE(f) \right|  \right]   \leq  \EMP\left[   \left|  \frac 4{L^2}\sum_x H_L(x)^2 U_T(x,x) - \SigGOE(f) \right| \right]
\\ +
  \frac 4{L^2}  \EMP\left[   \sum_x \sum_{k_1+k_2\geq 3}  \frac {| H_L(k_1x)   H_L(k_2x)| }{k_1  k_2 }
  |U_T\left(k_1x,k_2x \right)| \right] .
\end{multline}

\subsubsection{The sum $k_1+k_2\geq 3$}

 We set
 \[
D(k_1,k_2) :=   \frac 4{L^2}   \EMP\left[   \sum_x  \frac {| H_L(k_1x)   H_L(k_2x)| }{k_1  k_2 }
  |U_T\left(k_1x,k_2x \right)| \right]
 \]
 and want to show that
 \[
 \sum_{k_1+k_2\geq 3} D(k_1,k_2) \ll \frac{\log L}{L^2}  .
 \]

 We apply Campbell's formula (\eqref{eq:exp Camp gen} with $m=1$)
 \begin{equation}\label{campbell formula MP}
   \EMP\left[   \sum_x  h(x) \right] =\int_0^\infty h(x) d\nu_{MP}(x)
 \end{equation}
 with $h(x) =  \frac {| H_L(k_1x)   H_L(k_2x)| }{k_1  k_2 }   |U_T\left(k_1x,k_2x \right)| $
  to find
  \[
  \begin{split}
  D(k_1,k_2) &= \frac{4}{L^2}\int_0^\infty  \frac {| H_L(k_1x)   H_L(k_2x)| }{k_1  k_2 }   |U_T\left(k_1x,k_2x \right)| d\nu_{MP}(x)
  \\
  &\ll   \int_0^{C_0 L} \frac{\sinh^2(Lx/2)}{\sinh(Lk_1x/2)\sinh(Lk_2x/2)}  x  dx
  \end{split}
  \]
 on using boundedness of $U_T$ and $\^f$ being supported in $[-C_0,C_0]$. It is easy to show
 (see the proof of \cite[Lemma 5.2]{RGOE})
 that  this last integral is bounded by
 \[
  \ll \min \left( \frac 1{L^2(k_1+k_2-2)^2}, \frac 1{k_1k_2^3} \right)
 \]
 (assuming $k_2\geq k_1$) and from this to show  that
 \[
 \sum_{k_1+k_2\geq 3} D(k_1,k_2) \ll \frac{\log L}{L^2}
 \]
which is our claim.

 \subsubsection{The term $k_1=k_2=1$}
 We are left with showing
 that the first term on the RHS of \eqref{separate k1+k2 geq 3} tends to zero,
 and using Cauchy-Schwartz it suffices to bound the second moment
 \[
 \lim_{L\to \infty}     \EMP\left[ \left|   \frac 4{L^2}\sum_x H_L(x)^2 U_T(x,x) - \SigGOE(f) \right|^2 \right] = 0
 \]
for $T\geq 1$.  Squaring out, we want to show
 \begin{multline*}
  \EMP\left[ \left(  \frac 4{L^2} \sum_x H_L(x)^2U_T(x,x) \right)^2 \right]
  \\
   -2\SigGOE(f)   \EMP\left[ \frac 4{L^2}\sum_x H_L(x)^2U_T(x,x) \right]
  \to  (\SigGOE(f))^2   .
 \end{multline*}

We have
\begin{equation}\label{first sum is sigGOE}
 \EMP\left[\frac 4{L^2} \sum_x H_L(x)^2 U_T(x,x)\right] = \SigGOE(f) +O\left( \frac 1{TL} \right).
\end{equation}
Indeed, by Campbell's formula \eqref{campbell formula MP},
 \begin{multline*}
\EMP\left[\frac 4{L^2} \sum_x H_L(x)^2 U_T(x,x)\right]   =\frac 4{L^2} \int_0^\infty  H_L(x)^2 U_T(x,x) d\nu_{MP}(x)
\\
 =  \frac 4{L^2} \int_0^\infty   \frac{x^2\^f\left(\frac xL \right)}{\sinh^2(x/2)}  \frac 12\left(1+\^w(2Tx)-2\^w(Tx)^2 \right)2\frac{\sinh^2(x/2)}{x}dx
 \\
 = 4\int_0^\infty \^f(y)^2 y dy+  4\int_0^\infty \^f(y)^2 y \left( \^w(2TL y) - 2\^w(TLy)^2\right)dy ,
 \end{multline*}
using
\[
U_T(x,x) = \frac 12 + \frac 12 \^w(2Tx) -\^w(Tx)^2  .
\]

The first term is $2\intinf \^f(y)^2|y|dy = \SigGOE(f)$. The second term is treated by observing that since $\supp \^w\subseteq [-1,1]$, the integral is bounded by
\[
\int_0^\infty \^f(y)^2 y \left( |\^w(2TL y)| + 2|\^w(TLy)^2| \right) dy\ll \int_0^{1/2TL} \^f(y)^2 ydy \ll \frac 1{TL} ,
\]
hence we obtain \eqref{first sum is sigGOE}

Therefore it suffices to show
\[
 \EMP\left[ \left(  \frac 4{L^2}\sum_x H_L(x)^2 U_T(x,x)\right)^2 \right] \to  (\SigGOE(f))^2  .
 \]
 Again square out, obtain a diagonal sum and an off-diagonal sum.

 We bound the diagonal sum using Campbell's formula \eqref{campbell formula MP}
 \begin{multline*}
  \EMP\left[ \frac{16}{L^4} \sum_x H_L(x)^4  U_T(x,x)^2 \right] = \frac{16}{L^4}\int_0^\infty H_L(x)^4 U_T(x,x)^2
  d\nu_{MP}(x)
  \\
    =\frac{16}{L^4} \int_0^\infty \frac{x^4 \^f\left(\frac xL \right)^4}{\sinh^4(x/2)} \frac{2\sinh^2(x/2)}{x} U_T(x,x)^2  dx
  \\
   \ll   \frac 1{L^4} \int_0^{cL} \frac{ x^3}{\sinh^2(x/2)}dx \ll \frac 1{L^4} .
\end{multline*}

We evaluate the off-diagonal term using the correlation formula \eqref{eq:exp Camp gen} with $m=2$ for the Poisson process
  \begin{multline*}
  \EMP\left[\frac{16}{L^4} \sum_{x\neq y} H_L(x)^2  U_T(x,x)  H_L(y)^2  U_T(y,y) 2 \right]
  \\
  =\frac{16}{L^4}\iint H_L(x)^2 U_T(x,x)  H_L(y)^2U_T(y,y)     d\nu_{MP}(x)d\nu_{MP}(y)
 \\
  =
   \left( \frac 4{L^2}\int_0^\infty H_L(x)^2 U_T(x,x)  d\nu_{MP}(x) \right)^2 = \left( \SigGOE(f) \right)^2 + O\left(\frac 1{TL}\right)
  \end{multline*}
  by \eqref{first sum is sigGOE}, so we are done.

\section{Almost sure GOE fluctuations: Concluding the proof of Theorem \ref{thm:almost sure GOE}}

\begin{proof}

Let $\epsilon>0$ be given, and let us consider for a moment the random variables $\Vc_{T,L}=\Vc_{T,L}(X)$ and $\Vc_{T,L}^{\infty}$ for sufficiently large, but fixed $T,L$ within the allowed range $L=o(\log{T})$. It would follow that, as $g\rightarrow\infty$,
\begin{equation}
\label{eq:prob(|V-Sigma^2)>eps->}
\Prob\left( \left|\Vc_{T,L}- \Sigma_{GOE}^{2}(f)\right|>\epsilon \right)\rightarrow  \Prob\left( \left|\Vc_{T,L}^{\infty}- \Sigma_{GOE}^{2}(f)\right|>\epsilon \right),
\end{equation}
provided that $\Sigma_{GOE}^{2}(f)\pm \epsilon$ are not atoms of the distribution of $\Vc_{T,L}^{\infty}$. Otherwise, apply \eqref{eq:prob(|V-Sigma^2)>eps->} with
$$\epsilon/2<\epsilon'=\epsilon'(T,L)<\epsilon$$
so that, for the given $T,L$, the distribution of $\Vc_{T,L}^{\infty}$ does not have an atom at $\Sigma_{GOE}^{2}(f)\pm \epsilon'$, and use that, for sufficiently large admissible $T,L$, the probability of $$\left\{\left|\Vc_{T,L}^{\infty}- \Sigma_{GOE}^{2}(f)\right|>\epsilon'\right\} \subseteq \left\{\left|\Vc_{T,L}^{\infty}- \Sigma_{GOE}^{2}(f)\right|>\frac{\epsilon}{2}\right\}$$ is arbitrarily small, by \eqref{eq:VTL Pois close GOE}, a corollary from Theorem \ref{thm:VTL conv constant}.
Finally, the double limit
statement \eqref{eq:dev>eps -> 0 dbl lim} follows, as
\begin{equation*}
\begin{split}
&\Prob\left( \left|\Vc_{T,L}- \Sigma_{GOE}^{2}(f)\right|>\epsilon \right) \le
\Prob\left( \left|\Vc_{T,L}- \Sigma_{GOE}^{2}(f)\right|>\epsilon' \right) \\&=
\Prob\left( \left|\Vc_{T,L}^{\infty}- \Sigma_{GOE}^{2}(f)\right|>\epsilon' \right) \\&+  \left(\Prob\left( \left|\Vc_{T,L}- \Sigma_{GOE}^{2}(f)\right|>\epsilon' \right)  -\Prob\left( \left|\Vc_{T,L}^{\infty}- \Sigma_{GOE}^{2}(f)\right|>\epsilon' \right)\right)
\\&\le
\Prob\left( \left|\Vc_{T,L}^{\infty}- \Sigma_{GOE}^{2}(f)\right|>\frac{\epsilon}{2} \right) \\&+  \left(\Prob\left( \left|\Vc_{T,L}- \Sigma_{GOE}^{2}(f)\right|>\epsilon' \right)  -\Prob\left( \left|\Vc_{T,L}^{\infty}- \Sigma_{GOE}^{2}(f)\right|>\epsilon' \right)\right) ,
\end{split}
\end{equation*}
with the first term small by  \eqref{eq:VTL Pois close GOE}, and the difference tends to zero as $g\to \infty$ since $\Sigma_{GOE}^{2}(f)\pm \epsilon'$ is not an atom of $\Vc_{T,L}^{\infty}$.
\end{proof}


\begin{thebibliography}{99}


\bibitem{Berry1986}
Berry, M. V. Fluctuations in numbers of energy levels. Stochastic processes in classical and quantum systems (Ascona, 1985), 47–53, Lecture Notes in Phys., 262, Springer, Berlin, 1986.

\bibitem{BGGS}
Bogomolny, E.B., Georgeot, B., Giannoni, M.J. and Schmit, C. Chaotic billiards generated by arithmetic groups. Physical review letters, 69(10), p.1477 (1992).

\bibitem{Grandell} Grandell, J. Point processes and random measures. Advances in Applied Probability, 9(3), pp.502--526 (1977).

\bibitem{Kallenberg} Kallenberg, O. Random measures, theory and applications (Vol. 1). Cham: Springer International Publishing (2017).

\bibitem{LaPe} Last, G. and Penrose, M. Lectures on the Poisson process (Vol. 7). Cambridge University Press (2017).

\bibitem{LS} Luo, W. and Sarnak, P. Quantum ergodicity of Eigenfunctions on $ PSL_2 (\mathbf {Z})/\mathbf {H}^ 2$. Publications Math\'{e}matiques de l'IH\'{E}S, 81, pp.207--237 (1995).

\bibitem{Naud}
Naud, F. Random covers of compact surfaces and smooth linear spectral statistics. arXiv:2209.07941 [math.SP]


\bibitem{MP}
Mirzakhani, M. and Petri, B.
Lengths of closed geodesics on random surfaces of large genus.  Comment. Math. Helv. 94 (2019), no. 4, 869--889.

\bibitem{Pandey}
Pandey, A.
Statistical properties of many-particle spectra. III. Ergodic behavior in random-matrix ensembles.
Ann. Physics 119 (1979), no. 1, 170--191.

\bibitem{RudnickCLT}
Rudnick Z. A central limit theorem for the spectrum of the modular group, Annales Henri Poincare 6 (2005), 863--883.


\bibitem{RGOE}
Rudnick, Z.  GOE statistics on the moduli space of surfaces of large genus. arXiv:2202.06379 [math.SP]


\bibitem{RWCLT}
Rudnick, Z. and Wigman, I. On the Central Limit Theorem for linear eigenvalue statistics on random surfaces of large genus. arXiv:2301.00685 [math.SP]









\end{thebibliography}
\end{document}